\documentclass[12pt]{article}  
\usepackage{amsmath}
\usepackage{amssymb}
\usepackage{theorem}
\usepackage{euscript}
\usepackage{pstricks}
\usepackage{authblk}
\usepackage{verbatim}
\usepackage{graphics}
\usepackage{graphicx}
\usepackage{color}

\usepackage{hyperref}
\hypersetup
{
	colorlinks,
	citecolor=black,
	filecolor=black,
	linkcolor=blue,
	urlcolor=blue
}
\hypersetup{linktocpage}
\newtheorem{theorem}{Theorem}[section]
\newtheorem{lemma}[theorem]{Lemma}

\numberwithin{equation}{section}

\theoremstyle{definition}
 
\theoremstyle{remark}
\newtheorem{remark}[theorem]{Remark}

\newcommand{\brac}[1]{\left(#1\right)}

\newcommand{\brab}[1]{\left\{#1\right\}}

\newcommand{\bk}{{\boldsymbol{k}}}

\newcommand{\br}{{\boldsymbol{r}}}
\newcommand{\bs}{{\boldsymbol{s}}}
\newcommand{\bt}{{\boldsymbol{t}}}
\newcommand{\bx}{{\boldsymbol{x}}}

\newcommand{\bF}{{\boldsymbol{F}}}
\newcommand{\bX}{{\boldsymbol{X}}}
\newcommand{\by}{{\boldsymbol{y}}}

\newcommand{\bW}{{\boldsymbol{W}}}

\newcommand{\bone}{{\boldsymbol{1}}}

\newcommand{\bxi}{{\boldsymbol{\xi}}}
\newcommand{\rd}{{\rm d}} 
\newcommand{\Int}{{\rm Int}}

\def\ZZd{{\mathbb Z}^d}
\def\IId{{\mathbb I}^d}

\def\RR{{\mathbb R}}
\def\RRd{{\mathbb R}^d}

\def\NN{{\mathbb N}}
\def\NNd{{\NN}^d}

\def\NN{{\mathbb N}}
\def\RR{{\mathbb R}}

\def\TT{{\mathbb T}}
\def\IId{{\mathbb I}^d}

\def\NNd{{\mathbb N}^d}
\def\RRd{{\mathbb R}^d}

\def\ZZd{{\mathbb Z}^d}

\def\Hh{{\mathcal H}}

\def\Qq{{\mathcal Q}}

\def\NN{{\mathbb N}}
\def\RR{{\mathbb R}}

\def\TT{{\mathbb T}}
\def\IId{{\mathbb I}^d}

\def\NNd{{\mathbb N}^d}
\def\RRd{{\mathbb R}^d}

\def\sign{\operatorname{sign}}

\def\Wap{W^r_p}

\newcommand{\norm}[2]{\left\|{#1}\right\|_{#2}}

\title{\sffamily Numerical weighted integration of functions having mixed smoothness}

\author[a]{Dinh D\~ung}
\affil[a]{Information Technology Institute, Vietnam National University, Hanoi
	\protect\\
	144 Xuan Thuy, Cau Giay, Hanoi, Vietnam
	\protect\\
	Email: dinhzung@gmail.com}

\date{\today}
\begin{document}
\maketitle

\begin{abstract}
 We investigate the  approximation of weighted integrals over $\mathbb{R}^d$  for  integrands from weighted  Sobolev spaces of mixed smoothness. We prove upper and lower bounds  of the convergence rate of optimal quadratures with respect to $n$ integration nodes for functions from  these spaces.  In the one-dimensional case $(d=1)$,  we obtain the right convergence rate of optimal quadratures. For $d \ge 2$, the upper bound is performed by sparse-grid quadratures with integration nodes on step hyperbolic crosses in the function domain $\mathbb{R}^d$.
	
	
	
\end{abstract}

\section{Introduction}
\label{Introduction}

The aim of the present paper is to investigate  approximation of weighted integrals over $\mathbb{R}^d$  for   integrands lying in weighted  Sobolev spaces $W^r_{1,w}(\mathbb{R}^d)$ of  mixed smoothness $r \in \mathbb{N}$. 
We want to give upper and lower bounds of the approximation error for optimal quadratures with $n$ integration nodes  over the unit ball  of $W^r_{1,w}(\mathbb{R}^d)$. 

We first introduce weighted  Sobolev spaces of  mixed smoothness.  
Let 
\begin{equation} \label{w(bx)}
	w(\bx):= w_{\lambda,a,b}(\bx) := \prod_{i=1}^d w(x_i),
\end{equation}
where
\begin{equation} \label{w(x)}
	w(x):= 	w_{\lambda,a,b}(x) := \exp (- a|x|^\lambda + b), \ \ \lambda > 1, \ \ a > 0, \ \ b \in \RR,
\end{equation} 
is a univariate Freud-type weight. The most important parameter in the weight $w_{\lambda,a,b}$ is $\lambda$. The parameter $b$ which produces only a possitive constant in the weight $w_{\lambda,a,b}$  is introduced for a certain normalization for instance, for the standard Gaussian weight which is one of the most important weights. In what follows,  we fix the  parameters $\lambda,a,b$ and  for simplicity, drop them from the notation.
Let  $1\leq p<\infty$ and $\Omega$ be a Lebesgue measurable set on $\RRd$. 
We denote by  $L_{w}^p(\Omega)$ the weighted space  of all functions $f$ on $\Omega$ such that the norm
\begin{align} \label{L-Omega}
\|f\|_{L_{w}^p(\Omega)} : = 
\bigg( \int_\Omega |f(\bx)w(\bx)|^p  \rd \bx\bigg)^{1/p} 
\end{align}
is finite. For $r \in \NN$, we define the weighted  Sobolev space $W^r_{p,w}(\Omega)$ of mixed smoothness $r$  as the normed space of all functions $f\in L_{w}^p(\Omega)$ such that the weak (generalized) partial derivative $D^{\bk} f$ belongs to $L_{w}^p(\Omega)$ for  every $\bk \in \NNd_0$ satisfying the inequality $|\bk|_\infty \le r$. The norm of a  function $f$ in this space is defined by
\begin{align} \label{W-Omega}
	\|f\|_{W^r_{p,w}(\Omega)}: = \Bigg(\sum_{|\bk|_\infty \le r} \|D^{\bk} f\|_{L_{w}^p(\Omega)}^p\Bigg)^{1/p}.
\end{align}
It is useful to notice that any function $f \in W^r_{p,w}(\RRd)$ is equivalent  in the sense of the Lesbegue measure to a continuous (not necessarily bounded) function on $\RRd$, see Lemma~\ref{lemma:continuity} below. Hence throughout the present paper,  we always assume that the functions $f \in W^r_{p,w}(\RRd)$ are  continuous. We need this assumption for  well-defined  quadratures for functions $f \in W^r_{p,w}(\RRd)$.

Let $\gamma$ be the standard $d$-dimensional Gaussian measure $\gamma$ with the density function 
$$
g(\bx) = (2\pi)^{-d/2}\exp (- |\bx|_2^2/2).
$$
The well-known  spaces $L^p(\Omega;\gamma)$ and $\Wap(\Omega; \gamma)$ 
which are used in many applications, are defined  in the same way by replacing the norm \eqref{L-Omega} with the norm
$$
\|f\|_{L^p(\Omega; \gamma)} : = 
\bigg( \int_\Omega |f(\bx)|^p \gamma(\rd \bx)\bigg)^{1/p}
=
\bigg( \int_\Omega |f(\bx)|^p  g(\bx)\rd \bx\bigg)^{1/p}.
$$
The spaces $L^p(\Omega;\gamma)$ and $\Wap(\Omega; \gamma)$ can be seen as the   $L_{w}^p(\Omega)$ and $W^r_{p,w}(\Omega)$,
where
$$
w(\bx)=w_{\lambda ,a,b}(\bx), \ \  \text{with} \ \ \lambda = 2/p, \  a= 1/2p,  \ b= -(d \log 2\pi)/2p
$$
for  a fixed $1 \le p < \infty$.

In the present paper, we are interested in approximation of   weighted integrals 
\begin{equation} \label{If}
\int_{\RRd} f(\bx) w(\bx) \, \rd\bx 
\end{equation}
for functions $f$ lying  in the space 
$W^r_{1,w}(\mathbb{R}^d)$.
To approximate them we use quadratures  of the form
\begin{equation} \label{Q_nf-introduction}
	Q_kf: = \sum_{i=1}^k \lambda_i f(\bx_i), 
\end{equation}
where $\bx_1,\ldots,\bx_k \in \RRd$  are  the integration nodes and $\lambda_1,\ldots,\lambda_k$ the integration weights.  For convenience, we assume that some of the integration nodes may coincide. 

Let $\bF$ be a set of continuous functions on $\RRd$.  Denote by $\Qq_n$  the family of all quadratures $Q_k$ of the form \eqref{Q_nf-introduction} with $k \le n$. The optimality  of  quadratures from $\Qq_n$ for  $f \in \bF$  is measured by 
\begin{equation} \label{Int_n}
\Int_n(\bF) :=\inf_{Q_n \in \Qq_n} \ \sup_{f\in \bF} 
\bigg|\int_{\RRd} f(\bx) w(\bx) \, \rd\bx - Q_nf\bigg|.
\end{equation}

We recall that the space $\Wap(\Omega)$ is defined  as the classical Sobolev space of mixed smoothness by replacing $L_{w}^p(\Omega)$ with $L^p(\Omega)$ in \eqref{W-Omega}, where as usually,   $L^p(\Omega)$ denotes the Lebesgue space of functions on $\Omega$ equipped with the usual $p$-integral norm. 

For approximation of integrals 
$$
\int_\Omega f(\bx) \rd\bx
$$
over the set $\Omega$, we need natural modifications $Q_n^\Omega f$ 	for functions $f$ on $\Omega$,  and $\Int_n^\Omega(\bF)$  for  a set $\bF$ of  functions on  $\Omega$, of the definitions \eqref{Q_nf-introduction} and \eqref{Int_n}. For simplicity we will drop $\Omega$ from these notations if there is no misunderstanding.

We first briefly describe the main results of the present paper and then give comments on related works.

For a normed space $X$ of functions on $\RRd$, the boldface $\bX$ denotes the unit ball in $X$.  Throughout the present paper we make use of  the notation
$$r_\lambda:= (1 - 1/\lambda)r.$$ 
For  the set $\bW^r_{1,w}(\RRd))$, we prove the upper and lower bounds 
\begin{equation}\label{Int_n(W^r_1)}
	n^{-r_\lambda} (\log n)^{r_\lambda(d-1)} 
	\ll	
	\Int_n(\bW^r_{1,w}(\RRd)) 
	\ll 
	n^{-r_\lambda} (\log n)^{(r_\lambda + 1)(d-1)},
\end{equation}
in particular, in the case of Gaussian measure
\begin{equation}\label{Int_n(W^r_1)-tau=2}
	n^{-r/2} (\log n)^{r(d-1)/2} 
	\ll	
	\Int_n(\bW^r_1(\RRd; \gamma)) 
	\ll 
	n^{-r/2} (\log n)^{(r/2 + 1)(d-1)}.
\end{equation}
In  the one-dimensional  case, we prove the right convergence rate
\begin{equation}\label{Int_n(W)-introduction-d=1}
	\Int_n(\bW^r_{1,w}(\RR)) 
	\asymp 
	n^{-r_\lambda}.
\end{equation}
The difference between the upper and lower bounds in \eqref{Int_n(W^r_1)} is the logarithmic factor 
$(\log n)^{d-1}$.

There is a large number of works on high-dimensional unweighted integration over the unit $d$-cube $\IId:=[0,1]^d$ for functions having a mixed smoothness (see  \cite{DTU18B, DKS13, Tem18B} for results and bibliography). However, there are  only a few works  on high-dimensional weighted integration for functions having a mixed smoothness. The problem of  optimal  weighted integration \eqref{If}--\eqref{Int_n} has been studied in \cite{IKLP2015, IL2015,DILP18}  for functions in certain Hermite spaces, in particular, the space $\Hh_{d,r}$ which coincides with $W^r_2(\mathbb{R}^d;\gamma)$ in terms of norm equivalence.
It has been proven in \cite{DILP18} that
\begin{equation*}\label{DILP18}
	n^{-r}  (\log n)^{(d-1)/2}
	\ll	
	\Int_n\big( \bW^r_2(\RRd; \gamma))  
	\ll 
	n^{-r} (\log n)^{d(2r + 3)/4 - 1/2}.
\end{equation*}
Recently, in \cite[Theorem 2.3]{DK2022} for the space $W^r_p(\mathbb{R}^d, \gamma)$ with  $r\in \NN$ and $1<p<\infty$, we have constructed  an  asymptotically optimal quadrature  $Q_n^\gamma$ of the form \eqref{Q_nf-introduction} which  gives the  asymptotic order  
\begin{equation} 	\label{AsympQuadrature}
	\sup_{f\in \bW^r_p(\RRd; \gamma)} \bigg|\int_{\RRd}f(\bx) \gamma(\rd\bx) - Q_n^\gamma f\bigg| 
	\asymp
	\Int_n\big(\bW^r_p(\RRd; \gamma) \big) 
	\asymp
	n^{-r} (\log n)^{(d-1)/2}.
\end{equation}
The results \eqref{Int_n(W^r_1)-tau=2}  and \eqref{AsympQuadrature} show a substantial difference of the convergence rates between the cases $p=1$ and $1 < p < \infty$.
In constructing the asymptotically optimal quadrature $Q_n^\gamma$ in \eqref{AsympQuadrature}, we used a technique collaging a quadrature for the Sobolev spaces on the unit $d$-cube to the  integer-shifted $d$-cubes. Unfortunately, this technique is not suitable to constructing a quadrature   realizing the upper bound in \eqref{Int_n(W^r_1)} for the space  
$W^r_1(\mathbb{R}^d; \gamma)$ which is the largest among the spaces $W^r_p(\mathbb{R}^d; \gamma)$ with $1\le p < \infty$. It requires  a different technique based on the well-known Smolyak algorithm. Such a quadrature relies on  sparse grids of  integration nodes which are step hyperbolic crosses in the function domain $\RRd$,  and some generalization of the results on univariate numerical integration by truncated Gaussian quadratures from \cite{DM2003}. To prove the lower bound in \eqref{Int_n(W^r_1)} and \eqref{Int_n(W)-introduction-d=1} we adopt a traditional technique to construct for arbitrary $n$ integration nodes a fooling function vanishing at these nodes. 

It is interesting to compare the results  \eqref{Int_n(W^r_1)-tau=2}  and \eqref{AsympQuadrature} on	$\Int_n\big(\bW^r_p(\RRd; \gamma)\big)$ with known results  on
$\Int_n\big(\bW^r_p(\IId)\big)$ for the unweighted Sobolev space $W^r_p(\IId)$ of mixed smoothness $r$. For  $1<p<\infty$, there holds the asymptotic order
\begin{equation*}\label{Int_n-p>1-unweighted}
	\Int_n\big(\bW^r_p(\IId)\big) 
	\asymp 
	n^{-r} (\log n)^{(d-1)/2},
\end{equation*}
and for $p=1$ and $r > 1$, there hold the  bounds
\begin{equation*}\label{Int_n(W)IId}
	n^{-r} (\log n)^{(d-1)/2} 
	\ll	
	\Int_n(\bW^r_1(\IId)) 
	\ll 
	n^{-r} (\log n)^{d-1}
\end{equation*}
which are so far  the best known (see, e.g., \cite[Chapter 8]{DTU18B}, for detail).	
Hence we can see that $\Int_n\big(\bW^r_p(\RRd; \gamma)\big)$ and $\Int_n\big(\bW^r_p(\IId)\big)$ have the same asymptotic order  in the case 	$1<p<\infty$, and very different lower and upper bounds in both power and   logarithmic terms in the case $p=1$.
The right asymptotic orders of the both $\Int_n\big(\bW^r_1(\IId)\big)$ and $\Int_n\big(\bW^r_1(\RRd; \gamma)\big)$ are still open problems (cf. \cite[Open Problem 1.9]{DTU18B}).

The problem of numerical integration considered in the present paper is  related to  the research direction of optimal approximation and integration for functions having mixed smoothness on one hand, and the other research  direction of univariate weighted polynomial approximation and integration on $\RR$, on the other hand.  For survey and bibliography, we refer the reader to the books  \cite{DTU18B,Tem18B} on the first direction, and \cite{Mha1996B,Lu07B,JMN2021} on the second one.

The paper is organized as follows. In Section \ref{Univariate integration}, we prove the asymptotic order 
 of $\Int_n(\bW^r_{1,w}(\RR))$ and construct asymptotically optimal quadratures. In Section \ref{Multivariate integration}, we prove upper and lower bounds of $\Int_n(\bW^r_{1,w}(\RRd))$ for $d \ge 2$, and construct quadratures which give the upper bound. Section \ref{Extension} is devoted to some extentions of the results in the previous sections to Markov-Sonin weights.

\medskip
\noindent
{\bf Notation.} 
 Denote  $\bone:= (1,...,1) \in \RRd$; for $\bx \in \RRd$, $\bx=:\brac{x_1,...,x_d}$,
$|\bx|_\infty:= \max_{1\le j \le d} |x_j|$, $|\bx|_p:= \brac{\sum_{j=1}^d |x_j|^p}^{1/p}$ $(1 \le p < \infty)$.  For $\bx, \by \in \RRd$, the inequality $\bx \le \by$ means $x_i \le y_i$ for every $i=1,...,d$.  For $x \in \RR$, denote $\sign (x):= 1$ if $x \ge 0$, and $\sign (x):= -1$ if  $x < 0$. We use letters $C$  and $K$ to denote general 
positive constants which may take different values. For the quantities $A_n(f,\bk)$ and $B_n(f,\bk)$ depending on 
$n \in \NN$, $f \in W$, $\bk \in \ZZd$,  
we write  $A_n(f,\bk) \ll B_n(f,\bk)$, $f \in W$, $\bk \in \ZZd$ ($n \in \NN$ is specially dropped),  
if there exists some constant $C >0$ such that 
$A_n(f,\bk) \le CB_n(f,\bk)$ for all $n \in \NN$,  $f \in W$, $\bk \in \ZZd$ (the notation $A_n(f,\bk) \gg B_n(f,\bk)$ has the obvious opposite meaning), and  
$A_n(f,\bk) \asymp B_n(f,\bk)$ if $A_n(f,\bk) \ll B_n(f,\bk)$
and $B_n(f,\bk) \ll A_n(f,\bk)$.  Denote by $|G|$ the cardinality of the set $G$. 
For a Banach space $X$, denote by the boldface $\bX$ the unit ball in $X$.

	\section{One-dimensional integration}
\label{Univariate integration}

In this section, for one-dimensional numerical integration, we prove the asymptotic order of  $\Int_n\big(\bW^r_{1,w}(\RR)\big)$ and present some asymptotically optimal quadratures.
We start this section with a well-known inequality in the following lemma which is implied directly from the definition \eqref{Int_n} and which is quite useful  for lower estimation of $\Int_n(\bF)$.
 
\begin{lemma} \label{lemma:Int_n>}
Let $\bF$ be a set of continuous functions on $\RRd$. 	Then we have
\begin{equation} \label{Int_n>}
	\Int_n(\bF) \ge \inf_{\brab{\bx_1,...,\bx_n} \subset \RRd} \ \sup_{f\in \bF: \ f(\bx_i)= 0,\ i =1,...,n}\bigg|\int_{\RRd} f(\bx) w(\bx) \, \rd\bx\bigg|.
\end{equation}
\end{lemma}

%
	
We now consider the problem of approximation of  integral \eqref{If} for univariate functions from 
$W^r_{1,w}(\RR)$. Let $(p_m(w))_{m \in \NN}$ be the sequence of orthonormal polynomials with respect to the weight $w$. In the classical quadrature theory,  a possible choice of integration nodes is to take the zeros of the polynomials $p_m(w)$.
Denote by $x_{m,k}$, $1 \le k \le \lfloor m/2 \rfloor$ the positive zeros of 	$p_m(w)$, and by $x_{m,-k} = - x_{m,k}$ the negative ones (if $m$ is odd, then $x_{m,0}= 0$ is also a zero of $p_m(w)$). These zeros are located as 
\begin{equation}\label{zeros-location}
-a_m + \frac{Ca_m}{m^{2/3}} < x_{m,- \lfloor m/2 \rfloor} < \cdots 
< x_{m,-1} < x_{m,1} < \cdots <  x_{m,\lfloor m/2 \rfloor} \le a_m - \frac{Ca_m}{m^{2/3}}, 
\end{equation}
with a positive constant $C$ independent of $m$ (see, e. g., \cite[(4.1.32)]{JMN2021}). Here $a_m$ is the Mhaskar-Rakhmanov-Saff number which is
\begin{equation}\label{a_m}
	a_m = a_m(w)= (\gamma_\lambda m)^{1/\lambda}, \ \ \gamma_\lambda:= \frac{2 \Gamma((1+\lambda)/2)}{\sqrt{\pi} \Gamma(\lambda/2)},
\end{equation}
and $\Gamma$ is the gamma function. Notice that the formula \eqref{a_m} is given in \cite[(4.1.4)]{JMN2021} for the weight $w(x) = \exp \brac{-|x|^\lambda}$. Inspecting the definition of Mhaskar-Rakhmanov-Saff number (see, e.g., \cite[Page 116]{JMN2021}), one easily verify that it still holds true for the general weight $w$ for any $a >0$ and  $b\in \RR$.

 For a continuous function on $\RR$, the classical Gaussian quadrature is defined as
\begin{equation} \label{G-Q_mf}
	Q^{\rm G}_mf: = \sum_{|k| \le \lfloor m/2 \rfloor} \lambda_{m,k}(w) f(x_{m,k}), 
\end{equation} 
where $\lambda_{m,k}(w)$ are the corresponding Cotes numbers. This quadrature is based on Lagrange interpolation (for details, see, e.g., \cite[1.2. Interpolation and quadrature]{Mha1996B}). Unfortunately, it does not give the optimal convergence rate for functions from $\bW^r_{1,w}(\RR)$, see Remark \ref{Comment on G-quadrature} below. 

In \cite{DM2003}, for the weight $w(x) = \exp \brac{-|x|^\lambda}$, the authors proposed truncated Gaussian quadratures which not only improve the convergence rate but also give the asymptotic order of  $\Int_n\big(\bW^r_{1,w}(\RR)\big)$ as shown in Theorem \ref{thm:Q_n-d=1} below. Let us introduce  in the same manner truncated Gaussian quadratures for the weight 
$w(x)$ with any $a >0$ and  $b\in \RR$.

Throughout this paper, we fix a number $\theta$ with $0 < \theta < 1$, and denote by $j(m)$ the smallest integer satisfying $x_{m,j(m)} \ge \theta a_m$. 
It is useful to remark that
\begin{equation} \label{x_k+1 - x_k}
d_{m,k} \, \asymp  \, \frac{a_m}{m} \asymp m^{1/\lambda -1}, \ \ |k| \le j(m); \quad x_{m,j(m)} \, \asymp \, m^{1/\lambda},
\end{equation} 
where $d_{m,k}:= x_{m,k} - x_{m,k-1}$ is the distance between consecutive zeros of  the polynomial $p_m(w)$. These relations were proven  in \cite[(13)]{DM2003} for the weight $w(x) = \exp \brac{-|x|^\lambda}$. 
From their proofs there, one can easily see that they are still hold true for the general case of the weight $w$. By  \eqref{zeros-location} and \eqref{x_k+1 - x_k}, for $m$ sufficiently large we have that 
\begin{equation} \label{<j(m)<}
Cm \le j(m) \le m/2
\end{equation} 
with a positive constant $C$ depending on $\lambda, a, b$ and $\theta$ only. 

For a continuous function on $\RR$, consider 
the truncated Gaussian quadrature 
\begin{equation} \label{G-Q_mf-TG}
	Q^{\rm{TG}}_{2 j(m)}f: = \sum_{|k| \le j(m)} \lambda_{m,k}(w) f(x_{m,k}). 
\end{equation} 
Notice that the number $2j(m)$ of samples in the quadrature $Q^{\rm{TG}}_{2 j(m)}f$ is strictly smalller than $m$ -- the number of samples in the quadrature  $Q^{\rm G}_mf$.   However, due to \eqref{<j(m)<} it has the asymptotic order as $2 j(m)\asymp m$ when $m$ going to infinity.

\begin{theorem} \label{thm:Q_n-d=1}
	For any $n \in \NN$, let $m_n$ be the largest integer such that $2 j(m_n) \le n$. Then the quadratures	$Q^{\rm{TG}}_{2 j(m_n)} \in \Qq_n$, $n \in \NN$, are  asymptotically optimal  for $\bW^r_{1,w}(\RR)$ and
	\begin{equation}\label{Q_n-d=1}
		\sup_{f\in \bW^r_{1,w}(\RR)} \bigg|\int_{\RR}f(x) w(x) \rd x - Q^{\rm{TG}}_{2 j(m_n)}f\bigg| 
		\asymp
		\Int_n\big(\bW^r_{1,w}(\RR)\big) 
		\asymp 
		n^{- r_\lambda}.
	\end{equation}
\end{theorem}

\begin{proof} 
For $f \in W^r_{1,w}(\RR)$, there holds the inequality
	\begin{equation}\label{DellaVecchia2003}
	\bigg|\int_{\RR}f(x) w(x) \rd x - Q^{\rm{TG}}_{2 j(m)}f\bigg| 
	\le 
C\brac{m^{-(1 - 1/\lambda)r} \norm{f^{(r)}}{L^1_w(\RR)} + e^{-Km}\norm{f}{L^1_w(\RR)} } 
\end{equation}
with some constants $C$ and $K$ independent of $m$ and $f$. This inequality was proven in \cite[Corollary 4]{DM2003} for the weight $w(x) = \exp \brac{-|x|^\lambda}$.  Inspecting  the proof of \cite[Corollary 4]{DM2003}, one can easily see that this inequality is also true for a weight of the form \eqref{w(bx)} with any $a >0$ and $b \in \RR$. The inequality \eqref{DellaVecchia2003} implies the upper bound in \eqref{Q_n-d=1}:
	\begin{equation*}\label{Q_n-d=1(2)}
\Int_n\big(\bW^r_{1,w}(\RR)\big)  \le 
	\sup_{f\in \bW^r_{1,w}(\RR)} \bigg|\int_{\RR}f(x) w(x) \rd x - Q^{\rm{TG}}_{2 j(m_n)}f\bigg| 
	\ll
	n^{- r_\lambda}.
\end{equation*}

The lower bound in \eqref{Q_n-d=1} is already contained in Theorem \ref{theorem:Int_n(W)} below. Since its proof is much simpler for the case $d=1$, let us proccess it separately. In order to prove the lower bound in \eqref{Q_n-d=1} we will apply Lemma \ref{lemma:Int_n>}. Let $\brab{\xi_1,...,\xi_n} \subset \RR$ be arbitrary $n$ points. For a given $n \in \NN$, we put $\delta = n^{1/\lambda - 1}$ and $t_j = \delta j$, $j \in \NN_0$. Then there is $i \in \NN$ with 
$n + 1 \le  i \le 2n + 2$ such that the interval $(t_{i-1}, t_i)$ does not contain any point from the set $\brab{\xi_1,...,\xi_n}$.
Take a nonnegative function $\varphi \in C^\infty_0([0,1])$, $\varphi \not= 0$, and put
	\begin{equation*}\label{b_s}
b_0:= \int_0^1 \varphi(y) \rd y > 0, \quad b_s :=  \int_0^1 |\varphi^{(s)}(y)| \rd y, \ s = 1,...,r.
\end{equation*}		
Define the functions $g$ and $h$ on $\RR$ by
\begin{equation*}\label{g}
g(x):= 
\begin{cases}
\varphi(\delta^{-1}(x - t_{i-1})), & \ \ x \in (t_{i-1}, t_i), \\
0, & \ \ \text{otherwise},	
\end{cases}	
\end{equation*}
and
	\begin{equation*}\label{h}
		h(x):= (gw^{-1})(x).		
	\end{equation*}
Let us estimate the norm $\norm{h}{W^r_{1,w}(\RR)}$. For a given $k \in \NN_0$ with $0 \le k \le r$,  we have
	\begin{equation}\label{h^(s)}
	h^{(k)} = (gw^{-1})^{(k)} = \sum_{s=0}^k \binom{k}{s} g^{(k-s)}(w^{-1})^{(s)}. 		
\end{equation}
By a direct computation we find that   for $x \in \RR$,
	\begin{equation}\label{w^(s)}
(w^{-1})^{(s)}(x) = (w^{-1})(x) (\sign (x))^s\sum_{j=1}^s c_{s,j}(\lambda,a) |x|^{\lambda_{s,j}},
\end{equation}
where  $\sign (x):= 1$ if $x \ge 0$, and $\sign (x):= -1$ if  $x < 0$, 
\begin{equation}\label{lambda_{s,s}}
\lambda_{s,s} = s(\lambda - 1) > \lambda_{s,s-1} > \cdots > \lambda_{s,1} = \lambda - s,
\end{equation}
 and  $c_{s,j}(\lambda,a)$ are polynomials in the variables $\lambda$ and $a$ of degree at most $s$ with respect to each variable.
Hence, we obtain
	\begin{equation}\label{h^(s)w(x)}
	h^{(k)}(x) w(x)= \sum_{s=0}^k \binom{k}{s} g^{(k-s)}(x) (\sign (x))^s\sum_{j=1}^s c_{s,j}(\lambda, a) |x|^{\lambda_{s,j}} 		
\end{equation}
which implies that 
	\begin{equation}\nonumber
	\int_{\RR}|h^{(k)} w|(x) \rd x 
	\le  C \max_{0\le s \le k} \ \max_{1 \le j \le s}  \int_{t_{i-1}}^{t_i} |x|^{\lambda_{s,j}}|g^{(k-s)}(x)| \rd x.		
\end{equation}
From \eqref{lambda_{s,s}},  the inequality $n^{1/\lambda} \le x \le (2n + 2)n^{1/\lambda - 1}$ and 
\begin{equation*}\label{int-g^(k-s)}
  \int_{t_{i-1}}^{t_i} |g^{(k-s)}(x)| \rd x = b_{k-s} \delta^{-k+s+1} = 	b_{k-s} n^{(k-s-1)(1 - 1/\lambda)},
\end{equation*}
we derive
	\begin{equation*}\label{int-h^(s)w}
		\begin{aligned}
	\int_{\RR}|h^{(k)} w|(x) \rd x 
	&\le  C \max_{0\le s \le k}   \int_{t_{i-1}}^{t_i} |x|^{\lambda_{s,s}}|g^{(k-s)}(x)| \rd x
	\\
	& \le  C \max_{0\le s \le k}   \brac{n^{1/\lambda}}^{s(\lambda - 1)}
	\int_{t_{i-1}}^{t_i}|g^{(k-s)}(x)| \rd x
	\\& 
	\le  C \max_{0\le s \le k} n^{s(\lambda - 1)/\lambda}		 n^{(k-s-1)(1 - 1/\lambda)}
	\\
	&= C n^{(1-1/\lambda)(k-1)} \le  C n^{(1-1/\lambda)(r-1)}.		
	\end{aligned}
\end{equation*}
If we define 
	\begin{equation*}\label{h-bar}
	\bar{h}:=  C^{-1} n^{-(1-1/\lambda)(r-1)}h,
\end{equation*}
then $\bar{h}$ is nonnegative, 
$\bar{h} \in \bW^r_{1,w}(\RR)$, $\sup \bar{h} \subset (t_{i-1},t_i)$ and
	\begin{equation*}\label{int-h^(s)w2}
	\begin{aligned}	
		\int_{\RR}(\bar{h}w)(x) \rd x 
	&=  C^{-1} n^{-(1-1/\lambda)(r-1)}\int_{t_{i-1}}^{t_i}g(x) \rd x
	\\
	&
	=  C^{-1} n^{-(1-1/\lambda)(r-1)} b_0 \delta  \gg n^{-(1-1/\lambda)r}
	\end{aligned}
\end{equation*}
 Since the interval $(t_{i-1}, t_i)$ does not contain any point from the set $\brab{\xi_1,...,\xi_n}$, we have $\bar{h}(\xi_k) = 0$, $k = 1,...,n$. Hence, by Lemma \ref{lemma:Int_n>},
	\begin{equation}\nonumber
		\Int_n\big(\bW^r_{1,w}(\RR)\big) 	\ge  \int_{\RR}\bar{h}(x) w(x) \rd x
	 \gg  n^{- r_\lambda}.
\end{equation}
The lower bound in \eqref{Q_n-d=1} is proven.
	\hfill
\end{proof}

\begin{remark} \label{Comment on G-quadrature}
{\rm	
In the case when $w(x)= \exp (- x^2/2)$ is the Gaussian density, the truncated Gaussian quadratures $Q^{\rm{TG}}_{2 j(m)}$ in Theorem \ref{thm:Q_n-d=1} give
\begin{equation}\label{Q_n-d=1-tau=2}
		\sup_{f\in {\bW}^r_{1,w}(\RR)} \bigg|\int_{\RR}f(x) w(x) \rd x - Q^{\rm{TG}}_{2 j(m_n)}f\bigg| 
\asymp
	\Int_n\big({\bW}^r_{1,w}(\RR)\big) 
	\asymp 
	n^{-r/2}.
\end{equation}
On the other hand, for the full Gaussian quadratures $Q^{\rm G}_n$,  it has been proven in \cite[Proposition 1]{DM2003}	the convergence rate
\begin{equation}\nonumber
	\sup_{f\in \bW^1_{1,w}(\RR)}	\bigg|\int_{\RR}f(x) w(x) \rd x  - Q^{\rm G}_nf\bigg|
	\asymp
	n^{-1/6}
\end{equation}
which is much worse than the convergence rate of 
$	\Int_n\big(\bW^1_{1,w}(\RR)\big) \asymp 	n^{-1/2}$ as 
in \eqref{Q_n-d=1-tau=2} for $r = 1$.
}
\end{remark}


	\section{High-dimensional integration}
\label{Multivariate integration}

In this section, for high-dimensional numerical integration ($d \ge 2$), we prove upper and lower bounds of   $\Int_n\big(\bW^r_{1,w}(\RRd)\big)$ and construct  quadratures based on  step-hyperbolic-cross grids of integration nodes which give the upper bounds. To do this we need some auxiliary lemmata.

\begin{lemma} \label{lemma:continuity}
	Let $1\le p < \infty$. Then any function $f \in W^r_{p,w}(\RRd)$ is equivalent  in the sense of the Lebesgue measure to  a continuous function on $\RRd$.
\end{lemma}
\begin{proof}
We prove this lemma in the particular case when $p=1$, $r=1$  and $d=2$. The general case can be proven in a similar way. 

Fix $\tau >\lambda$ and define for $\bx=(x_1,x_2)$,
$$
v(\bx):= \exp(-a|x_1|^\tau  + b)\exp(-a|x_2|^\tau + b).
$$

We preliminarly prove that $W^r_{1,w}(\TT^2)$ is continuously embbeded  into the space $C_v(\TT^2)$ where $\TT^2:= [-T,T]^2$, $T$ is any positive number and $C_v(\TT^2)$ is the Banach space of continuous functions $f$ on $\TT^2$ equipped with the norm
$$
\norm{f}{C_v(\TT^2)}:= \max_{\bx \in \TT^2} |(vf)(\bx)|.
$$
Since the subspace $C^\infty_0(\TT^2)$	 of infinitely differentiable functions with compact support is dense in both the Banach spaces $C(\TT^2)$ and $W^1_{1,w}(\TT^2)$, to prove this continuous embbeding,  it is sufficient to show the inequality 
\begin{equation}\label{embbeding}
\norm{f}{C_v(\TT^2)}	\ll \norm{f}{W^1_{1,w}(\TT^2)}, \ \ f \in C^\infty_0(\TT^2).
\end{equation}
For $\bk \in \NN_0^2$, denote by $D^\bk g$ the $\bk$th partial derivative of $g$.	Taking a function $f \in C^\infty_0(\TT^2)$, we have that for $\bx \in \TT^2$,
$$
(vf)(\bx) = \int_{-T}^{x_1}\int_{-T}^{x_2} D^{(1,1)}(vf)(\bt) \rd \bt,
$$
and 
\begin{equation}\nonumber
	\begin{aligned}
		D^{(1,1)} (vf)(\bx)
		& =
		v(\bx)\big[ D^{(1,1)} f(\bx) - a \tau \sign(x_1)|x_1|^{\tau - 1} D^{(1,0)}f(\bx)
		\\
		&
		\ \ \ - a \tau \sign(x_2)|x_2|^{\tau - 1} D^{(0,1)}f(\bx)		
		+ a^2 \tau^2 \sign(x_1)|x_1|^{\tau - 1} \sign(x_2)|x_2|^{\tau - 1} f(\bx)\big].
	\end{aligned}
\end{equation}
Hence by using the inequality $\tau >\lambda$ we derive \eqref{embbeding}:
\begin{equation}\nonumber
	\begin{aligned}
		\norm{f}{C_v(\TT^2)}	
		& \ll 
		\int_{\TT^2} \big|\big(v D^{(1,1)} f\big)(\bx)\big| \rd \bx 
		+ \int_{\TT^2} \big|\big(v D^{(1,0)} f\big)(\bx)\big| |x_1|^{\tau - 1} \rd \bx
		\\
		&
		+\int_{\TT^2} \big|\big(v D^{(0,1)} f\big)(\bx)\big||x_2|^{\tau - 1}  \rd \bx  
		+\int_{\TT^2} \big|\big(v f\big)(\bx)\big||x_1x_2|^{\tau - 1}  \rd \bx
		\\ 
		&
		\ll\int_{\TT^2} \big|\big(w D^{(1,1)} f\big)(\bx)\big| \rd \bx 
		+ \int_{\TT^2} \big|\big(w D^{(1,0)} f\big)(\bx)\big| \rd \bx
		\\
		&
		+\int_{\TT^2} \big|\big(w D^{(0,1)} f\big)(\bx)\big| \rd \bx  
		+\int_{\TT^2} \big|\big(w f\big)(\bx)\big| \rd \bx = \norm{f}{W^1_{1,w}(\TT^2)}
	\end{aligned}
\end{equation}
From the continuous embbeding of $W^r_{1,w}(\TT^2)$ into $C_v(\TT^2)$ it follows that  
any function $f \in W^r_{1,w}(\TT^2)$ is equivalent  in sense of the Lebesgue measure to  a continuous (not necessarily bounded) function on $\TT^2$. Hence we obtain the claim of the lemma for $p = 1$ since $T$ has been taken arbitrarily and the restriction of  a function $f \in W^r_{1,w}(\RR^2)$ to $\TT^2$ belongs to $W^r_{1,w}(\TT^2)$.
\hfill
\end{proof}	

Importantly, as noticed in Introduction from Lemma \ref{lemma:continuity}   we can  assume that the functions $f \in W^r_{p,w}(\RRd)$ are  continuous. This allows to correctly define quadratures for them.

For $\bx \in \RRd$ and $e \subset \brab{1,...,d}$, let $\bx^e \in \RR^{|e|}$ be defined by $(x^e)_i := x_i$, and  $\bar{\bx}^e\in \RR^{d-|e|}$ by $(\bar{x}^e)_i := x_i$, $i \in \brab{1,...,d} \setminus e$. With an abuse we write 
$(\bx^e,\bar{\bx}^e) = \bx$.

\begin{lemma} \label{lemma:g(bx^e}
	Let $1\le p \le \infty$,  $e \subset \brab{1,...,d}$ and $\br \in \NNd_0$. Assume that $f$ is a function on $\RRd$  such that for every $\bk \le \br$, $D^\bk f \in L^p_w(\RRd)$. 
	Put for  $\bk \le \br$ and $\bar{\bx}^e \in \RR^{d-|e|}$,
	\begin{equation*}\label{g(bx^e}
	g(\bx^e): =  D^{\bar{\bk}^e} f(\bx^e,\bar{\bx}^e).
	\end{equation*}
Then $D^\bs g \in L^p_w(\RR^{|e|})$ for every $\bs \le \bk^{e}$ and almost every 
$\bar{\bx}^e \in \RR^{d-|e|}$.
\end{lemma}

\begin{proof}
Taking  arbitrary test functions $\varphi^e \in C^\infty_0(\RR^{|e|})$	and $\bar{\varphi}^e \in C^\infty_0(\RR^{d-|e|})$  and defining
$\varphi (\bx):= \varphi^e(\bx^e) \bar{\varphi}^e(\bar{\bx}^e)$,  we have that $\varphi \in C^\infty_0(\RRd)$.
For $\bk \le \br$ and $\bs \in \NNd_0$ with $s_i \le k_i$, $i \in e$ and $s_i := 0$ otherwise, we derive that
\begin{equation}\nonumber
	\begin{aligned}
		&\int_{\RR^{d-|e|}}\bar{\varphi}^e(\bar{\bx}^e) \int_{\RR^{|e|}}
		  g(\bx^e) D^\bs\varphi^e(\bx^e)\rd \bx^e \,\rd \bar{\bx}^e
		\\ & =
	\int_{\RR^{d-|e|}}\bar{\varphi}^e(\bar{\bx}^e) \int_{\RR^{|e|}}
	D^{\bar{\bk}^e} f(\bx^e,\bar{\bx}^e) D^\bs\varphi^e(\bx^e)\rd \bx^e \,\rd \bar{\bx}^e
	\\
	&=
		\int_{\RRd}
	D^{\bar{\bk}^e} f(\bx) D^\bs\varphi(\bx)\rd \bx
		= (-1)^{|\bs|_1}	\int_{\RRd}
		D^{\bar{\bk}^e + \bs} f(\bx) \varphi(\bx)\rd \bx
		\\
		&
		= 	\int_{\RR^{d-|e|}}\bar{\varphi}^e(\bar{\bx}^e) (-1)^{|\bs|_1}\int_{\RR^{|e|}}
		D^{\bar{\bk}^e +\bs} f(\bx^e,\bar{\bx}^e) \varphi^e(\bx^e)\rd \bx^e \,\rd \bar{\bx}^e.
	\end{aligned}
\end{equation}
Hence,
\begin{equation}\nonumber
	\begin{aligned}
	\int_{\RR^{|e|}}
		 g(\bx^e) D^\bs\varphi^e(\bx^e)\rd \bx^e 
		=
		 (-1)^{|\bs|_1}\int_{\RR^{|e|}}
		D^{\bar{\bk}^e +\bs} f(\bx^e,\bar{\bx}^e) \varphi^e(\bx^e)\rd \bx^e
	\end{aligned}
\end{equation}
for almost every $\bar{\bx}^e \in \RR^{d-|e|}$. 
 This means that the weak derivative $D^\bs g$ exists for almost every 
$\bar{\bx}^e \in \RR^{d-|e|}$ which coincides with $D^{\bar{\bk}^e + \bs} f(\cdot,\bar{\bx}^e)$. 
Moreover, $D^\bs g \in L^p_w(\RR^{|e|})$ for almost every $\bar{\bx}^e \in \RR^{d-|e|}$ since 
by the assumption $D^\bk f \in L^p_w(\RRd)$ for every $\bk \le \br$.
 \hfill
\end{proof}

	Assume that there exists a sequence of quadratures $\brac{Q_{2^k}}_{k \in \NN_0}$ with
\begin{equation}\label{Q_2^kf}
	Q_{2^k}f: = \sum_{s=1}^{2^k} \lambda_{k,s} f(x_{k,s}), \ \ \{x_{k,1},\ldots,x_{k,2^k}\}\subset \RR,
\end{equation}
such that 
\begin{equation}\label{IntErrorR}
	\bigg|\int_{\RR} f(x) w(x)\rd x - Q_{2^k}f\bigg| 
	\leq C 2^{- ak} \|f\|_{W^r_{1,w}(\RR)}, 
	\ \  \ k \in \NN_0,  \ \ f \in W^r_{1,w}(\RR), 
\end{equation}
for some  number $a>0$ and constant $C>0$.

Based on a sequence $\brac{Q_{2^k}}_{k \in \NN_0}$ of the form \eqref{Q_2^kf} satisfying \eqref{IntErrorR}, we construct quadratures on $\RRd$ by using the well-known Smolyak algorithm. We define for $k \in \NN_0$, the one-dimensional operators
\begin{equation*}\label{DeltaI}
\Delta_k^Q:= Q_{2^k} - Q_{2^{k-1}}, \  k >0, \ \ \Delta_0^Q:= Q_1,  
\end{equation*}
and 
\begin{equation*}\label{E}
	E_k^Qf:= \int_{\RR} f(x) w(x)\rd x - Q_{2^k}f.
\end{equation*}
For $\bk \in \NNd$, the $d$-dimensional operators $Q_{2^\bk}$,  $\Delta_\bk^Q$ and $E_\bk^Q$ are defined as the tensor  product of one-dimensional operators:
\begin{equation}\label{tensor-product}
	Q_{2^\bk}:= \bigotimes_{i=1}^d Q_{2^{k_i}} , \ \	
	\Delta_\bk^Q:= \bigotimes_{i=1}^d \Delta_{k_i}^Q, \ \ 	
	E_\bk^Q:= \bigotimes_{i=1}^d E_{k_i}^Q, 
\end{equation}
where $2^\bk:= (2^{k_1},\cdots, 2^{k_d})$ and 
the univariate operators $Q_{2^{k_j}}$, $\Delta_{k_j}^Q$ and $E_{k_j}^Q$ 
 are successively applied to the univariate functions $\bigotimes_{i<j} Q_{2^{k_i}}(f)$, $\bigotimes_{i<j} \Delta_{k_i}^Q(f)$ and $\bigotimes_{i<j} E_{k_i}^Q $, respectively, by considering them  as 
functions of  variable $x_j$ with the other variables held fixed. The operators $Q_{2^\bk}$, $\Delta_\bk^Q$ and $E_\bk^Q$ are well-defined for continuous functions on $\RRd$, in particular for ones from $W^r_{1,w}(\RRd)$.

Notice that if $f$ is a continuous function on $\RRd$, then  $Q_{2^\bk} f$ is a quadrature on $\RRd$ which is given by
\begin{equation}\label{Q_2^bkf}
	Q_{2^\bk}f = \sum_{\bs=\bone}^{2^\bk} \lambda_{\bk,\bs} f(\bx_{\bk,\bs}), 
	\ \ \{\bx_{\bk,\bs}\}_{\bone \le \bs \le 2^\bk}\subset \RRd,
\end{equation}
where $$
\bx_{\bk,\bs}:= \brac{x_{k_1,s_1},...,x_{k_d,s_d}}, \ \ \ 
\lambda_{\bk,\bs}:= \prod_{i=1}^d \lambda_{k_i,s_i} ,
$$
 and the summation $\sum_{\bs=\bone}^{2^\bk}$ means that the sum is taken over all $\bs$ such that $\bone \le \bs \le 2^\bk$. Hence we derive that 
\begin{equation}\label{Delta_bk}
\Delta_\bk^Q f =  \sum_{e \subset \brab{1,...,d}} (-1)^{d - |e|}Q_{2^{\bk(e)}} f
	= \sum_{e \subset \brab{1,...,d}} (-1)^{d - |e|}\sum_{\bs=\bone}^{2^{\bk(e)}} \lambda_{\bk(e),\bs} f(\bx_{\bk(e),\bs}), 
\end{equation}
where  $\bk(e) \in \NNd_0$ is defined by $k(e)_i = k_i$, $i \in e$, and 	$k(e)_i = \max(k_i-1,0)$, $i \not\in e$.  We also have
\begin{equation}\label{E_bk}
	E_\bk^Q f =  \sum_{e \subset \brab{1,...,d}} (-1)^{|e|} 
	\int_{\RR^{d - |e|}}Q_{2^{\bk^e}}f(\cdot,\bar{\bx^e} ) w(\bar{\bx}^e) \rd \bar{\bx}^e,  
\end{equation}
where $w(\bar{\bx}^e):= \prod_{j \not\in e} w(x_j)$.

Notice that as mappings from $C(\RRd)$ to $\RR$, the operators $Q_{2^\bk}$, $\Delta_\bk^Q$ and $E_\bk^Q$  possess commutative and associative properties  with respect to applying the component operators $Q_{2^{k_j}}$, $\Delta_{k_j}^Q$ and $E_{k_j}^Q$ in the following sense.  We have  for any $e \subset \brab{1,...,d}$,
$$
Q_{2^\bk}f= Q_{2^{\bk^e}}\brac{Q_{2^{\bar{\bk}^e}} f}, \ \ 
\Delta_\bk^Q f= \Delta_{\bk^e}^Q \brac{\Delta_{\bar{\bk}^e}^Q f}, \ \ 
E_\bk^Q f= E_{\bk^e}^Q \brac{E_{\bar{\bk}^e}^Q f},
$$
and for any reordered sequence $\brab{i(1),..., i(d)}$ of $\brab{1,...,d}$, 
\begin{equation}\label{commutative}
	Q_{2^\bk} = \bigotimes_{j=1}^d Q_{2^{k_{i(j)}}}, 
	\ \	\Delta_\bk^Q = \bigotimes_{j=1}^d \Delta_{k_{i(j)}}^Q, \ \ 	
	E_\bk^Q = \bigotimes_{j=1}^d E_{k_{i(j)}}^Q.
\end{equation}
  These properties directly follow from 
\eqref{Q_2^bkf}--\eqref{E_bk}.

\begin{lemma} \label{lemma:E_k}
	Under the assumption \eqref{Q_2^kf}--\eqref{IntErrorR}, we have
	\begin{equation*}\label{E_k}
		\big|E_\bk^Qf\big| 
		\leq C 2^{- a|\bk|_1} \|f\|_{W^r_{1,w}(\RRd)}, 
		\ \  \bk \in \NNd_0, \ \ f \in W^r_{1,w}(\RRd).
	\end{equation*}
\end{lemma}

\begin{proof} The case $d=1$ of the lemma is as in \eqref{IntErrorR} by the assumption. For simplicity we prove the lemma for the case $d=2$. The general case can be proven in the same way by induction on $d$. We make use of the temporary notation:
$$
\norm{f}{W^r_{1,w}(\RR), 2}(x_1):= \norm{f(x_1,\cdot)}{W^r_{1,w}(\RR)}. 
$$ 
From Lemmata \ref{lemma:continuity} and  \ref{lemma:g(bx^e} it follows that 	
$f(\cdot,x_2) \in W^r_{1,w}(\RR)$ for every $x_2 \in \RR$. Notice that $E_{k_2}^Qf$ is a function in the variable $x_1$ only. Hence, by \eqref{IntErrorR}
 we obtain  that
	\begin{equation}\nonumber
		\begin{aligned}
			\big|E_{(k_1,k_2)}^Qf\big| & =
		\big|E_{k_1}^Q (E_{k_2}^Qf)\big| \leq C 2^{- ak_1} \|E_{k_2}^Qf\|_{W^r_{1,w}(\RR)}
		\\
			&\leq C 2^{- ak_1} \|2^{- ak_2} \|f\|_{W^r_{1,w}(\RR),2}(\cdot)\|_{W^r_{1,w}(\RR)}
		= C 2^{- a|\bk|_1} \|f\|_{W^r_{1,w}(\RR^2)}.
	\end{aligned}
	\end{equation}
	\hfill
\end{proof}

We say that $\bk \to \infty$, $\bk \in \NNd_0$, if and only if $k_i \to \infty$ for every $i = 1,...,d$.
\begin{lemma} \label{lemma:Delta_k}
	Under the assumption \eqref{Q_2^kf}--\eqref{IntErrorR}, we have that
	for every $ f \in W^r_{1,w}(\RRd)$,
		\begin{equation}\label{Series1}
	\int_{\RRd} f(\bx) w(\bx)\rd \bx
	= 	\sum_{\bk \in \NNd_0}\Delta_\bk^Qf 
	\end{equation}
with absolute convergence of the series, and
	\begin{equation}\label{Delta_k}
		\big|\Delta_\bk^Qf\big| 
		\leq C 2^{- a|\bk|_1} \|f\|_{W^r_{1,w}(\RRd)}, 
		\ \  \bk \in \NNd_0.
	\end{equation}
\end{lemma}

\begin{proof}  The operator $\Delta_\bk^Q$ can be represented in the form
	\begin{equation}\nonumber
\Delta_\bk^Q
	= 	\sum_{e \subset \brab{1,...,d}} (-1)^{|e|} E_{\bk(e)}^Q.
\end{equation}	
Therefore, by using Lemma \ref{lemma:E_k} we derive that for every $f \in W^r_{1,w}(\RRd)$ and  $\bk \in \NNd_0$,
	\begin{equation}\nonumber
		\begin{aligned}
			\big|\Delta_\bk^Qf\big| 
			& \le
			\sum_{e \subset \brab{1,...,d}} \big|E_{\bk(e)}^Qf\big| 
			\\
			&\leq 
			\sum_{e \subset \brab{1,...,d}}  C 2^{- a|\bk(e)|_1} \|f\|_{W^r_{1,w}(\RRd)}
			\le C 2^{- a|\bk|_1} \|f\|_{W^r_{1,w}(\RRd)}
		\end{aligned}
	\end{equation}
which proves \eqref{Delta_k} and hence  the absolute convergence of the series in \eqref{Series1} follows.
 Notice that 
\begin{equation}\nonumber
	\int_{\RRd} f(\bx) w(\bx)\rd \bx - Q_{2^\bk}f
	= 	\sum_{e \subset \brab{1,...,d}, \ e \not= \varnothing} (-1)^{|e|} E_{\bk^e}^Qf,
\end{equation}	
where recall $\bk^e \in \NNd_0$ is defined by $k^e_i = k_i$, $i \in e$, and 	$k^e_i = 0$, $i \not\in e$.
By using Lemma~\ref{lemma:E_k} we derive  for $\bk \in \NNd_0$ and $f \in W^r_{1,w}(\RRd)$,
\begin{equation}\nonumber
	\begin{aligned}
		\bigg|\int_{\RRd} f(\bx) w(\bx)\rd \bx  -  Q_{2^\bk}f\bigg| 
		& \le
		\sum_{e \subset \brab{1,...,d}, \ e \not= \varnothing}\big|E_{\bk^e}^Qf\big| 
		\\
		&\leq 
		C \max_{e \subset \brab{1,...,d}, \ e \not= \varnothing} \ \max_{1 \le i \le d} 2^{- a|k^e_i|} \|f\|_{W^r_{1,w}(\RRd)}
	\\
	&\leq 
	C \max_{1 \le i \le d} 2^{- a|k_i|} \|f\|_{W^r_{1,w}(\RRd)},
	\end{aligned}
\end{equation}
which is going to $0$ when $\bk \to \infty$. This together with the obvious equality
	\begin{equation}\nonumber
Q_{2^\bk}
	= 	\sum_{\bs \le \bk} \Delta_{\bs}^Q
\end{equation}	
proves  \eqref{Series1}.
	\hfill
\end{proof}

We now define an  algorithm for quadrature on sparse grids adopted from the alogorithm for sampling recovery initiated by Smolyak (for detail see \cite[Sections 4.2 and 5.3]{DTU18B}). For $\xi > 0$, we define the operator
	\begin{equation}\nonumber
	Q_\xi
	:= 	\sum_{|\bk|_1 \le \xi } \Delta_{\bk}^Q.
\end{equation}	
From \eqref{Delta_bk} we can see that $Q_\xi$ is a quadrature on $\RRd$ of the form  \eqref{Q_nf-introduction}:
\begin{equation}\label{Q_xi}
	Q_\xi f
	= 	\sum_{|\bk|_1 \le \xi } \ \sum_{e \subset \brab{1,...,d}} (-1)^{d - |e|}\ \sum_{\bs=\bone}^{2^{\bk(e)}} \lambda_{\bk(e),\bs} f(\bx_{\bk(e),\bs})
	= \sum_{(\bk,e,\bs) \in G(\xi)} \lambda_{\bk,e,\bs} f(\bx_{\bk,e,\bs}), 
\end{equation}
where 
$$
\bx_{\bk,e,\bs}:= \bx_{\bk(e),\bs}, \quad \lambda_{\bk,e,\bs}:= (-1)^{d - |e|}\lambda_{\bk(e),\bs}
$$ 
and 
\begin{equation}\nonumber
G(\xi)	:= \brab{(\bk,e,\bs): \ |\bk|_1 \le \xi, \,   e \subset \brab{1,...,d}, \,  \bone \le \bs \le \bk(e)}
\end{equation}
is a finite set.
The set of integration nodes in this quadrature 
\begin{equation}\nonumber
H(\xi):=\brab{\bx_{\bk,e,\bs}}_{(\bk,e,\bs) \in G(\xi)}
\end{equation}
is a step hyperbolic cross   in the function domain $\RRd$. 
The number of integration nodes in the quadrature $Q_\xi$ is 
\begin{equation}\nonumber
	|G(\xi)|
	= 	\sum_{|\bk|_1 \le \xi } \ \sum_{e \subset \brab{1,...,d}}2^{|\bk(e)|_1}
\end{equation}
which can be estimated as 
\begin{equation}\label{|G(xi)|}
	|G(\xi)|
	\asymp	\sum_{|\bk|_1 \le \xi }2^{|\bk|_1} \ \asymp \ 2^\xi \xi^{d - 1}, \ \ \xi \ge 1.
\end{equation}
As commented in Introduction, this quadrature plays a crucial role in the proof of the upper bound in the main results of the present paper \eqref{Int_n(W^r_1)}.

\begin{lemma} \label{lemma:Q_xi-error}
	Under the assumption \eqref{Q_2^kf}--\eqref{IntErrorR}, we have that
\begin{equation}\label{upperbound}
	\bigg|\int_{\RRd} f(\bx) w(\bx)\rd \bx  -  Q_\xi f\bigg| 
	\leq C 2^{- a\xi} \xi^{d - 1} \|f\|_{W^r_{1,w}(\RRd)}, 
	\ \ \xi \ge 1, \ \ f \in W^r_{1,w}(\RRd). 
\end{equation}
\end{lemma}

\begin{proof}  From  Lemma \ref{lemma:Delta_k} we derive that for $\xi \ge 1$ and $f \in W^r_{1,w}(\RRd)$,
	\begin{equation}\nonumber
		\begin{aligned}
	\bigg|\int_{\RRd} f(\bx) w(\bx)\rd \bx  -  Q_\xi f\bigg| 
			& \le
				\sum_{|\bk|_1 > \xi } \big|\Delta_\bk^Qf\big| 
				\le 	C\sum_{|\bk|_1 > \xi } 2^{-a|\bk|_1}  \|f\|_{W^r_{1,w}(\RRd)}
		\\
		&\leq 
	C \|f\|_{W^r_{1,w}(\RRd)} \sum_{|\bk|_1 > \xi } 2^{-a|\bk|_1} 
	\leq C 2^{- a\xi} \xi^{d - 1} \|f\|_{W^r_{1,w}(\RRd)}.
		\end{aligned}
	\end{equation}
	\hfill
\end{proof}
\begin{remark} \label{remark1}
	{\rm	
		From Theorem \ref{thm:Q_n-d=1} we can see that the truncated Gaussian quadratures $Q^{\rm{TG}}_{2 j(m)}$ 
		form  a sequence $\brac{Q_{2^k}}_{k \in \NN_0}$ of the form \eqref{Q_2^kf} satisfying  \eqref{IntErrorR} with $a = r_\lambda$. 
	}
\end{remark}

\begin{remark}
	{\rm	
The technique for proving the upper bound \eqref{upperbound} is analogous to a general technique  for establishing upper bounds of the error of unweighted sampling recovery by Smolyak algorithms of functions having mixed smoothness  on a bounded domain (see, e.g., \cite[Section 5.3]{DTU18B} and \cite[Section 6.9]{TB18} for detail).		
	}
\end{remark}

\begin{theorem} \label{theorem:Int_n(W)}
	We have that
	\begin{equation}\label{Int_n(W)}
	n^{-r_\lambda} (\log n)^{r_\lambda(d-1)} 
	\ll	
	\Int_n(\bW^r_{1,w}(\RRd)) 
	\ll 
	n^{-r_\lambda} (\log n)^{(r_\lambda + 1)(d-1)}.
	\end{equation}
\end{theorem}

\begin{proof}  
Let us first prove the upper bound in \eqref{Int_n(W)}. We  will construct a quadrature of the form \eqref{Q_xi} which realizes it.	
In order to do this, we take the truncated Gaussian quadrature 	$Q^{\rm{TG}}_{2 j(m)}f$ defined in \eqref{G-Q_mf}. For every $k \in \NN_0$, let $m_k$ be the largest number such that $2j(m_k) \le 2^k$.  Then we have 
$2j(m_k) \asymp 2^k$. 
For the sequence of quadratures  $\brac{Q_{2^k}}_{k \in \NN_0}$ with 
$$
Q_{2^k}:= Q^{\rm{TG}}_{2 j(m_k)} \in \Qq_{2^k},
$$
 from Theorem \ref{thm:Q_n-d=1} it follows that
\begin{equation*}\label{IntErrorR2}
	\bigg|\int_{\RR} f(x) w(x)\rd x - Q_{2^k}f\bigg| 
	\leq C 2^{- r_\lambda k} \|f\|_{W^r_{1,w}(\RR)}, 
	\ \  \ k \in \NN_0,  \ \ f \in W^r_{1,w}(\RR), 
\end{equation*}
This means that the assumption \eqref{Q_2^kf}--\eqref{IntErrorR} holds for $a = r_\lambda$. To prove the upper bound in \eqref{Int_n(W)} we approximate the integral 
$$
\int_{\RRd} f(\bx) w(\bx)\rd \bx
$$
  by the quadrature $Q_{\xi}$ which is formed from the sequence $\brac{Q_{2^k}}_{k \in \NN_0}$.
For every $n \in \NN$, let $\xi_n$ be the largest number such that $|G(\xi_n)| \le n$.  Then the corresponding operator $Q_{\xi_n}$ defines a quadrature belonging to $\Qq_n$. From \eqref{|G(xi)|} it follows 
$$
\ 2^{\xi_n} \xi_n^{d - 1} \asymp |G(\xi_n)|  \asymp n.
$$ 
Hence we deduce the asymptotic equivalences
$$
\ 2^{-\xi_n}  \asymp n^{-1} (\log n)^{d-1}, \ \  \xi_n \asymp \log n,
$$
which together with Lemma \ref{lemma:Q_xi-error}  yield that
	\begin{equation*}\label{Q_xi-error}
		\begin{aligned}
		\Int_n(\bW^r_{1,w}(\RRd)) 
		&\le 
	\sup_{f\in \bW^r_{1,w}(\RRd)}	
	\bigg|\int_{\RRd} f(\bx) w(\bx)\rd \bx  -  Q_{\xi_n}f\bigg| 
		\\
		&
		\leq 
		C 2^{- r_\lambda \xi_n} \xi_n^{d - 1}
		\asymp  	n^{-r_\lambda} (\log n)^{(r_\lambda + 1)(d-1)}.
		\end{aligned}
	\end{equation*}
The upper bound in \eqref{Int_n(W)} is proven.

We now prove the lower bound  in \eqref{Int_n(W)} by using the inequality \eqref{Int_n>} in Lemma \ref{lemma:Int_n>}. For $M \ge 1$, we define the set 
	\begin{equation*}\label{Gamma-set}
\Gamma_d(M):= \brab{\bs \in \NNd: \, \prod_{i=1}^d s_i \le 2M, \ s_i \ge M^{1/d}, \ i=1,...,d}.
\end{equation*}
Then we have 
	\begin{equation}\label{|Gamma|}
	|\Gamma_d(M)| \asymp M (\log M)^{d-1}, \ \ M>1.
\end{equation}
Indeed, we have the inclusion
$$\Gamma_d(M) \subset \Gamma'_d(M):= \brab{\bs \in \NNd: \, \prod_{i=1}^d s_i \le  2M}$$ and 
$$
|\Gamma'_d(M)| \asymp M (\log M)^{d-1}.
$$ 
Hence, 
$|\Gamma_d(M)| \ll M (\log M)^{d-1}.$ 
We prove the converse inequality 
$
|\Gamma_d(M)| \gg M (\log M)^{d-1}
$
by induction on the dimension $d$. It is obvious for $d = 1$. Assuming that this inequality  is true for $d-1$,  we check 
it for $d$, $(d \ge 2)$. Fix a positive number $\tau$ with $1 < \tau < d$. We have by induction assumption,
\begin{equation}\nonumber
	\begin{aligned}
			|\Gamma_d(M)|
		& = \sum_{M^{1/d} \le s_d \le 2M}    |\Gamma_{d-1}(2Ms_d^{-1})|
		\gg \sum_{M^{1/d} \le s_d \le 2M}  
		\brac{2Ms_d^{-1}} \brac{\log  2Ms_d^{-1}}^{d-2}
		\\
		& \gg M \sum_{M^{1/d} \le s_d \le 2M^{\tau/d}}  
		s_d^{-1} \brac{\log  2Ms_d^{-1}}^{d-2}
		\\
		&\ge 
		M \sum_{M^{1/d} \le s_d \le 2M^{\tau/d}}  
		s_d^{-1} \brac{\log  2M^{1 - \tau/d}}^{d-2}
		\\
		&\gg M \brac{\log  M}^{d-2}\sum_{M^{1/d} \le s_d \le 2M^{\tau/d}} s_d^{-1}
		\,  \gg \, M \brac{\log  M}^{d-1}.
	\end{aligned}
\end{equation}
The asymptotic equivalence \eqref{|Gamma|} is proven.

For a given $n \in \NN$, let $\brab{\bxi_1,...,\bxi_n} \subset \RRd$ be arbitrary  $n$ points. Denote by $M_n$  the smallest number such that $|\Gamma_d(M_n)| \ge n + 1$. We define the $d$-parallelepiped $K_\bs$ for $\bs \in \NNd_0$ of size 
$$
\delta:= M_n^{\frac{1/\lambda - 1}{d}}
$$
 by
	\begin{equation*}\label{K_bs}
K_\bs:= \prod_{i=1}^d K_{s_i}, \ \ K_{s_i}:= (\delta s_i, \delta s_{i-1}).
\end{equation*}
 Since  $|\Gamma_d(M_n)| > n$, there exists a multi-index 
$\bs \in \Gamma_d(M_n)$ such that $K_\bs$ does not contain any point from $\brab{\bxi_1,...,\bxi_n}$.

As in the proof of Theorem \ref{thm:Q_n-d=1}, we take a nonnegative function $\varphi \in C^\infty_0([0,1])$, $\varphi \not= 0$, and put
\begin{equation}\label{b_s-d}
	b_0:= \int_0^1 \varphi(y) \rd y > 0, \quad b_s :=  \int_0^1 |\varphi^{(s)}(y)| \rd y, \ s = 1,...,r.
\end{equation}		
For $i= 1,...,d$, we define the univariate functions $g_i$ in variable $x_i$ by
\begin{equation}\label{g_i}
	g_i(x_i):= 
	\begin{cases}
		\varphi(\delta^{-1}(x_i - \delta s_{i-1})), & \ \ x_i  \in K_{s_i}, \\
		0, & \ \ \text{otherwise}.	
	\end{cases}	
\end{equation}
Then the mulitivariate functions $g$ and $h$ on $\RRd$ are defined by
	\begin{equation*}\label{g(bx)}
	g(\bx):= \prod_{i=1}^d g_i(x_i),
\end{equation*}
and  
\begin{equation}\label{h(bx)}
	h(\bx):= (gw^{-1})(\bx)= \prod_{i=1}^d g_i(x_i)w^{-1}(x_i)=:
	\prod_{i=1}^d h_i(x_i).		
\end{equation}
Let us estimate the norm $\norm{h}{W^r_{1,w}(\RRd)}$. For every $\bk \in \NNd_0$ with 
$0 \le |\bk|_\infty \le r$,  we prove the inequality
\begin{equation}\label{int-D^br}
	\int_{\RRd}\big|(D^{\bk} h) w\big|(\bx) \rd \bx 
	\le  C   M_n^{(1 - 1/\lambda)(r - 1)}.
\end{equation}
We have
\begin{equation}\label{D^bk h}
	D^\bk h  = \prod_{i=1}^d h_i^{(k_i)}. 		
\end{equation}
Similarly to \eqref{h^(s)}--\eqref{h^(s)w(x)} we derive that for every $i = 1,...,d$,
\begin{equation*}\label{h^(s-i)w(x_i)}
	h_i^{(k_i)}(x_i) w(x_i)= \sum_{\nu_i=0}^{k_i} \binom{k_i}{\nu_i} g_i^{(k_i- \nu_i)}(x_i) (\sign (x_i))^{\nu_i}\sum_{\eta_i=1}^{\nu_i} c_{\nu_i,\eta_i}(\lambda,a) |x_i|^{\lambda_{\nu_i,\eta_i}}, 		
\end{equation*}
where
\begin{equation*}\label{lambda_{nu,nu}}
	\lambda_{\nu_i,\nu_i} = \nu_i(\lambda - 1) > \lambda_{\nu_i,\nu_i-1} > \cdots > \lambda_{\nu_i,1} = \lambda - \nu_i,
\end{equation*}
and  $c_{\nu_i,\eta_i}(\lambda,a)$ are polynomials in the variables $\lambda$ and $a$ of degree at most $\nu_i$ with respect to each variable.
This together with \eqref{b_s-d}--\eqref{g_i} and the inequalities $s_i \ge M_n^{\frac{1}{d}}$ and $\lambda_{\nu_i,\nu_i} = \nu_i(\lambda - 1) \ge  0$  yields that 
\begin{equation}\label{int-h^(s)w-d>1}
	\begin{aligned}
			\int_{\RR}\big|h_i^{(k_i)}(x_i) w(x_i)\big|\rd x_i
	&\le  C \max_{0\le \nu_i \le k_i} \ \max_{1 \le \eta_i \le \nu_i}  
	\int_{K_{s_i}}|x_i|^{\lambda_{\nu_i,\eta_i}}\big|g^{(k_i-\nu_i)}(x_i)\big| \rd x_i 
	\\
	&\le  C \max_{0\le \nu_i \le k_i} (\delta s_i)^{\lambda_{\nu_i,\nu_i}}
	\int_{K_{s_i}}\big|g^{(k_i-\nu_i)}(x_i)\big| \rd x_i 
	\\
	&\le C \max_{0\le \nu_i \le k_i} (\delta s_i)^{\nu_i(\lambda - 1)}
	\delta^{- k_i + \nu_i +1} b_{k_i - \nu_i}
	\\
	&= C \delta^{- k_i +1} \max_{0\le \nu_i \le k_i} \brac{\delta^\lambda  s_i^{\lambda -1}}^{\nu_i}. 
\end{aligned}
\end{equation}
Since $s_i \ge M_n^{\frac{1}{d}}$ and $\delta:= M_n^{\frac{1/\lambda - 1}{d}}$, we have that
$\delta^\lambda  s_i^{\lambda -1} \ge 1$, and consequently, 
$$
\max_{0\le \nu_i \le k_i} \brac{\delta^\lambda  s_i^{\lambda -1}}^{\nu_i}
 = 
 \brac{\delta^\lambda  s_i^{\lambda -1}}^{k_i}.
$$
This equality,  the estimates \eqref{int-h^(s)w-d>1} and the inequalities  $0 \le k_i \le r$ and 
$\delta s_i \ge 1$ yield that
\begin{equation}\nonumber
	\begin{aligned}
		\int_{\RR}\big|h_i^{(k_i)}(x_i) w(x_i)\big|\rd x_i
		&\le  C \delta^{- k_i +1} \brac{\delta^\lambda  s_i^{\lambda -1}}^{k_i} 
		= C \delta \brac{\delta s_i}^{k_i(\lambda -1)}
		\\ 
		& \le  C \delta \brac{\delta s_i}^{r(\lambda -1)} 
		=  C \delta ^{r(\lambda - 1) + 1} s_i^{r(\lambda - 1)} 
	\end{aligned}
\end{equation}
Hence, by \eqref{D^bk h} we deduce 
\begin{equation}\nonumber
	\begin{aligned}
		\int_{\RRd}|(D^\bk h) w|(\bx) \rd \bx 
		&= \prod_{i=1}^d  \int_{\RR}\big|h^{(k_i)}(x_i) w(x_i)\big|\rd x_i
		\\
		&  \le  C \prod_{i=1}^d  \delta ^{r(\lambda - 1) + 1}  s_i^{r(\lambda - 1)} 
		 \le  C \delta ^{d(r(\lambda - 1) + 1)} \brac{\prod_{i=1}^d  s_i}^{r(\lambda - 1)}.
\end{aligned}
\end{equation}
Since  $\prod_{i=1}^d  s_i \le 2M_n$, $\delta:= M_n^{\frac{1/\lambda - 1}{d}}$ and $\lambda  > 1$, we can continue the estimation as
\begin{equation*}\label{int-D^br2}
	\begin{aligned}
		\int_{\RRd}|(D^\bk h) w|(\bx) \rd \bx 
		\le  C  M_n^{(r(\lambda - 1) + 1)(1/\lambda - 1)}   M_n^{r(\lambda - 1)} 
			= C   M_n^{(1 - 1/\lambda)(r - 1)},
	\end{aligned}
\end{equation*}
which completes the proof of the inequality \eqref{int-D^br}.
This  inequality means that $h \in W^r_{1,w}(\RRd)$ and  
$$
\norm{h}{W^r_{1,w}(\RRd)} \le  C   M_n^{(1 - 1/\lambda)(r - 1)}.
$$
 If  we define 
\begin{equation*}\label{h-bar}
	\bar{h}:=  C^{-1} M_n^{-(1-1/\lambda)(r-1)}h,
\end{equation*}
then $\bar{h}$ is nonnegative, $\bar{h} \in \bW^r_{1,w}(\RR)$, $\sup \bar{h} \subset K_\bs$ and by \eqref{b_s-d}--\eqref{h(bx)},
\begin{equation}\nonumber
	\begin{aligned}	
		\int_{\RRd}(\bar{h}w)(\bx) \rd \bx 
		&=  C^{-1} M_n^{-(1-1/\lambda)(r-1)}\int_{\RRd}(h w)(\bx) \rd \bx = \prod_{i=1}^d \int_{K_{s_i}} g_i(x_i) \rd x_i
		\\
		&
		=  C^{-1} M_n^{-(1-1/\lambda)(r-1)} \brac{b_0 \delta}^d =  C' M_n^{- r_\lambda}.
	\end{aligned}
\end{equation}
From the definition of $M_n$ and \eqref{|Gamma|} it follows that
$$M_n(\log M_n)^{d-1} \asymp |\Gamma(M_n)| \asymp n,$$ 
 which implies that $M_n^{-1} \asymp n^{-1} (\log n)^{d-1}$. This allows to receive the estimate
\begin{equation}\label{int-h-bar2}
		\int_{\RRd}(\bar{h}w)(\bx) \rd \bx 
		=  C' M_n^{- r_\lambda}
		\gg
		n^{-r_\lambda} (\log n)^{r_\lambda(d-1)}.
\end{equation}
 Since the interval $K_\bs$ does not contain any point from the set $\brab{\bxi_1,...,\bxi_n}$ which has been arbitrarily choosen, we have 
 $$\bar{h}(\bxi_k) = 0, \ \ k = 1,...,n.
 $$
  Hence, by  Lemma \ref{lemma:Int_n>} and \eqref{int-h-bar2} we have that
\begin{equation*}\label{int-h-bar3}
	\Int_n(\bW^r_{1,w}(\RRd)) 
	\ge
	 \int_{\RRd}\bar{h}(\bx) w(\bx) \rd \bx
	\gg  n^{-r_\lambda} (\log n)^{r_\lambda(d-1)}.
\end{equation*}
The lower bound in \eqref{Int_n(W)} is proven.
	\hfill
\end{proof}

\begin{remark}
	{\rm	
		Let us analyse some properties of the quadratures $Q_\xi$ and their integration nodes $H(\xi)$ which give the upper bound in \eqref{Int_n(W)}. 
		\\
		\\
		1. The set of integration nodes $H(\xi)$ in the quadratures $Q_\xi$ which are formed from the non-equidistant zeros of the orthonornal polynomials $p_m(w)$,
		is a  step hyperbolic cross   \emph{on the function domain} $\RRd$. This is a contrast to the  classical theory of  approximation of multivariate periodic functions having mixed smoothness for which the classical step hyperbolic crosses of integer points are \emph{on the frequency domain} $\ZZd$ (see, e.g., \cite[Section 2.3]{DTU18B} for detail). The terminology 'step hyperbolic cross' of intergration nodes is borrowed from there. In Figure \ref{Fig-1}, in particular, the step hyperbolic cross in the right picture is designed for the Hermite weight $w(\bx)= \exp(- x_1^2 - x_2^2)$ ($d=2$). The set  $H(\xi)$ also completely differs from the classical Smolyak grids of fractional dyadic points \emph{on the function domain} $[-1,1]^d$ (see Figure \ref{Fig-2} for $d=2$) which are used in sparse-grid sampling recovery and numerical integration for functions having a mixed smoothness (see, e.g., \cite[Section~5.3]{DTU18B} for detail). 
		 \\ 
		 \\
		2. The set  $H(\xi)$ is  very sparsely distributed inside the $d$-cube 
		$$K(\xi): = \brab{\bx \in \RRd: \, |x_i| \le C2^{\xi/\lambda}, \ i =1,...,d},$$
		for some constant $C > 0$.
		 Its diameter  which is the length of its symmetry axes is  $2C2^{\xi/\lambda}$, i.e.,  the size of $K(\xi)$. 
		The number of integration nodes in $H(\xi)$ is  $|G(\xi)|\asymp  2^\xi \xi^{d - 1}$. For the integration nodes $H(\xi)=\brab{\bx_{\bk,e,\bs}}_{(\bk,e,\bs) \in G(\xi)}$, we have that
	\begin{equation}\nonumber
		\min_{\substack{(\bk,e,\bs), (\bk',e',\bs') \in G(\xi) \\ (\bk,e,\bs) \not= (\bk',e',\bs')} } 
		\ \min_{1 \le i \le d} \left|\brac{x_{\bk,e,\bs}}_i - \brac{x_{\bk',e',\bs'}}_i \right|  \ \asymp \  2^{- (1 - 1/\lambda)\xi} \ \to 0, \ \text{when} \ \xi \to \infty.
\end{equation}	
On the other hand, the diameter of $H(\xi) $ is going to $\infty$ when $\xi \to \infty$.
}
\end{remark}

\begin{figure}
	\centering{
	\begin{tabular}{cc}
		\includegraphics[height=7.0cm]{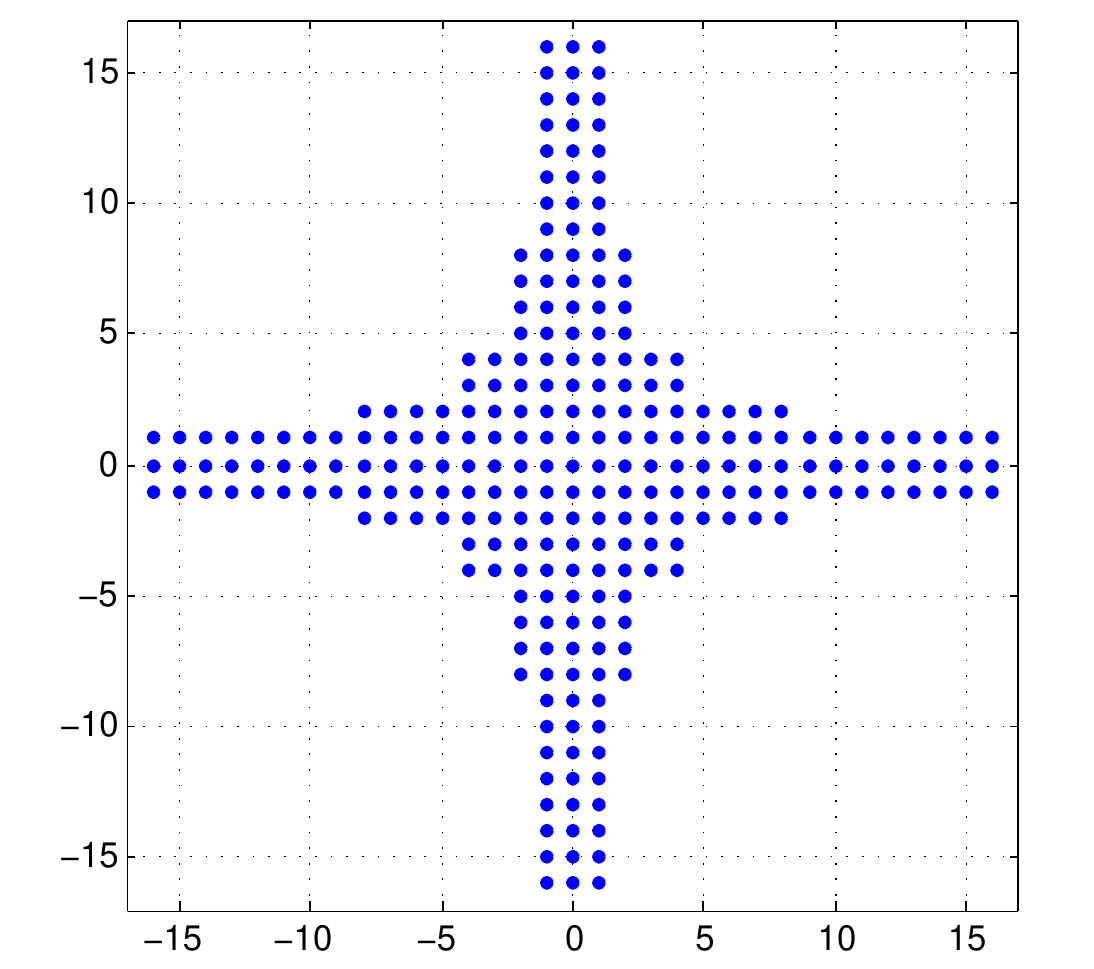}	 &  
		\includegraphics[height=7.0cm]{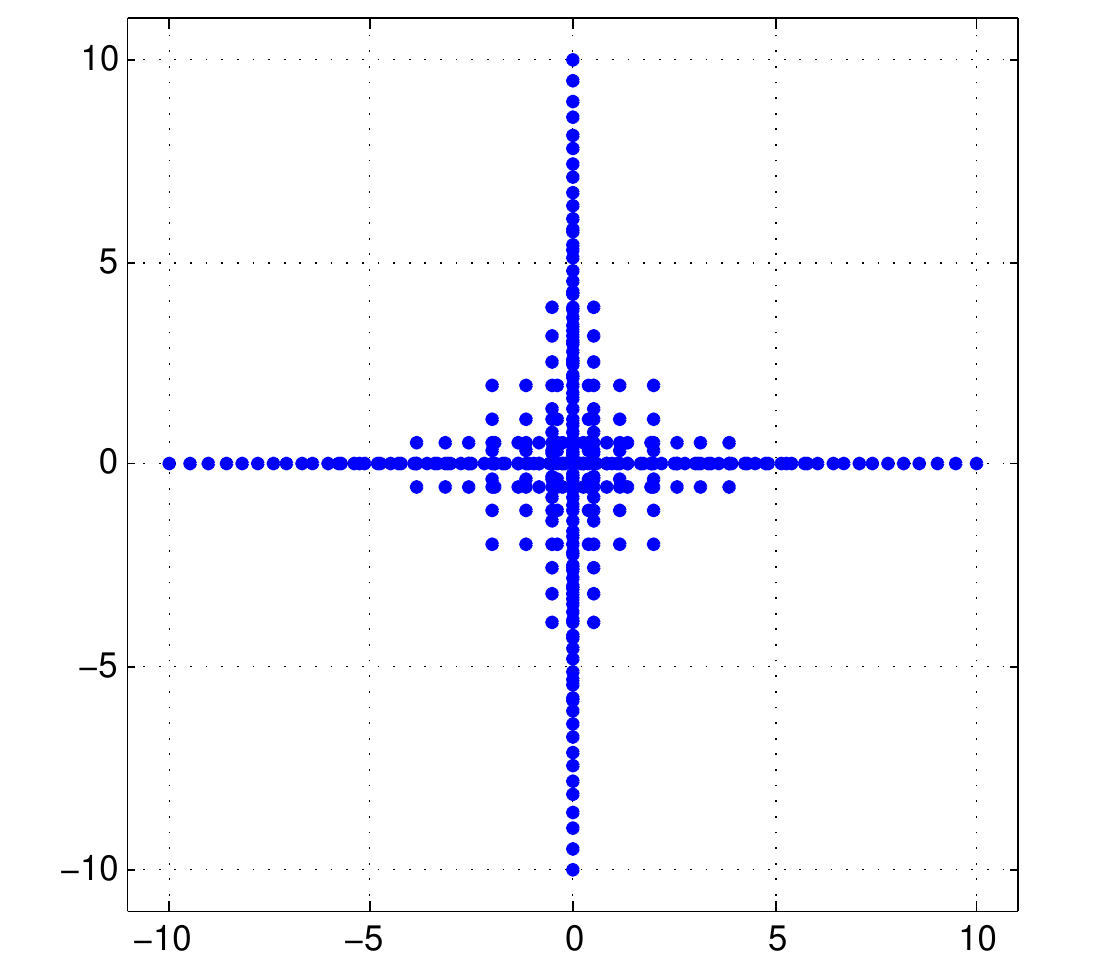}
		\\
		{A classical step hyperbolic cross} & { A Hermite step hyperbolic cross}
	\end{tabular}
}
	\caption{Step hyperbolic crosses ($d=2$)}
	\label{Fig-1}
\end{figure}

\begin{figure}
	\centering{
	\begin{tabular}{cc}
		\includegraphics[height=7.0cm]{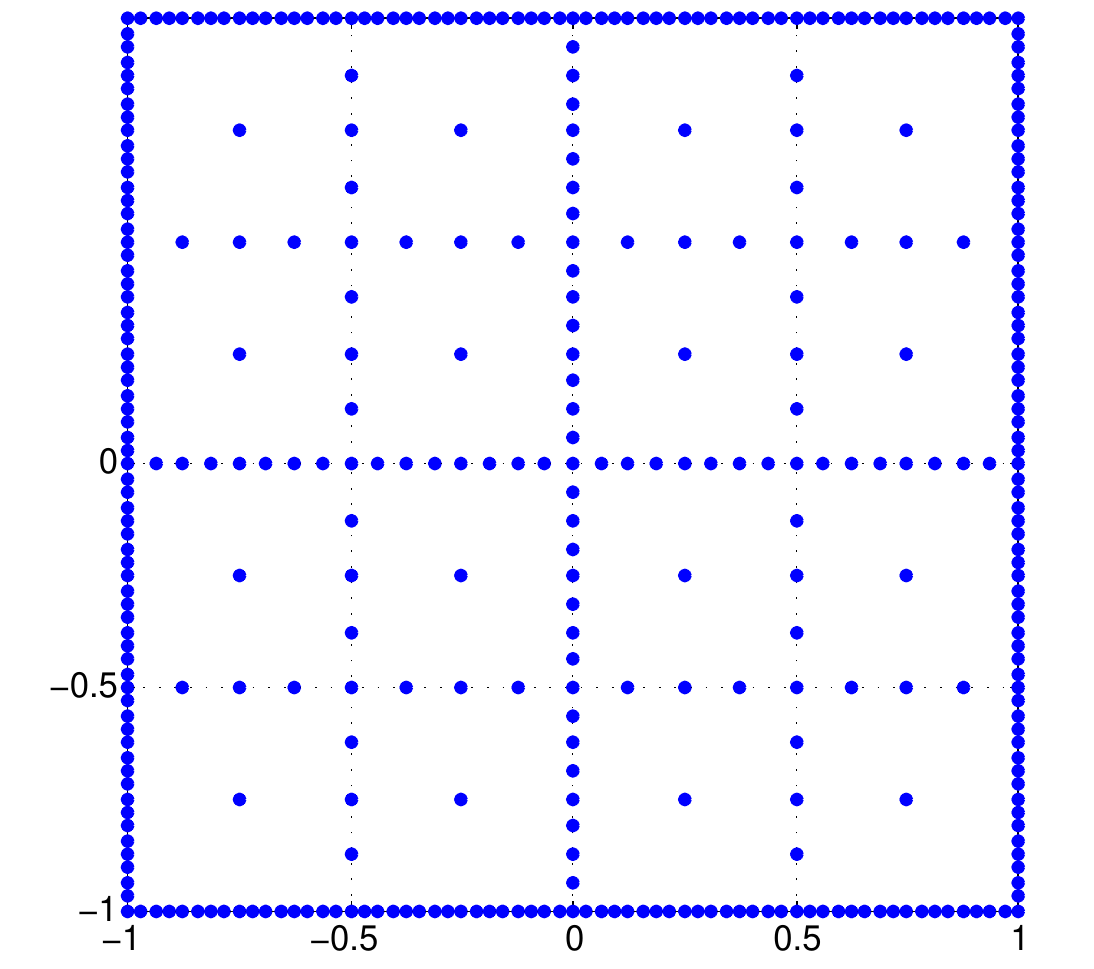}	
	\end{tabular}
}
	\caption{A Smolyak grid ($d=2$)}
	\label{Fig-2}
\end{figure}

\section{Extension to Markov-Sonin weights}
\label{Extension}

In this section, we extend the results of the previous sections to  Markov-Sonin weights. A univariate Markov-Sonin weight is a function of the form 
\begin{equation} \nonumber 
	w_\beta(x):= 	|x|^\beta \exp (- a|x|^2 + b), \ \ \beta > 0, \ \ a > 0, \ \ b \in \RR,
\end{equation}
(here $\beta$ is indicated in the notation to distinguish Markov-Sonin weights $w_\beta$ and Freud-type weight $w$).  
A $d$-dimensional Markov-Sonin weight is defined as
\begin{equation} \nonumber
	w_\beta(\bx):= \prod_{i=1}^d w_\beta(x_i).
\end{equation}
Markov-Sonin weights are not of the form \eqref{w(x)} and have a singularity at $0$. 
We will keep all the notations  and definitions in Sections 
\ref{Introduction}--\ref{Multivariate integration}  with replacing $w$ by $w_\beta$, pointing some modifications.

Denote $\mathring{\RR}^d:= \brac{\RR \setminus \brab{0}}^d$ and 
$\mathring{\Omega}:= \Omega \cap \mathring{\RR}^d$. Besides the spaces $L_{w_\beta}^p(\Omega)$ and $W^r_{p,w_\beta}(\Omega)$ we consider also the spaces $L_{w_\beta}^p\big(\mathring{\Omega}\big)$ and $W^r_{p,w_\beta}\big(\mathring{\Omega}\big)$ which are defined in a similar manner. For the space $W^r_{p,w_\beta}\big(\mathring{\Omega}\big)$, we require one of the following restrictions on $r$ and $\beta$ to be satisfied:
\begin{itemize}
	\item [{\rm (i)}]
	 $\beta > r - 1$;
	\item [{\rm (ii)}] 
	$0 <\beta < r - 1$ and $\beta$ is not an integer, for $f \in W^r_{p,w_\beta}\big(\mathring{\Omega}\big)$, the derivative $D^\bk f$ can be extended to a continuous function on $\Omega$ for all $\bk \in \NNd_0$ such that 
	$|\bk|_\infty \le r - 1 -\lceil \beta \rceil$. 
	\end{itemize}

Let $(p_m(w_\beta))_{m \in \NN}$ be the sequence of orthonormal polynomials with respest to the weight $w_\beta$.
Denote again by $x_{m,k}$, $1 \le k \le \lfloor m/2 \rfloor$ the positive zeros of 	$p_m(w_\beta)$, and by $x_{m,-k} = - x_{m,k}$ the negative ones (if $m$ is odd, then $x_{m,0}= 0$ is also a zero of $p_m(w_\beta)$). 
If $m$ is even, we add $x_{m,0} := 0$.
These nodes are located as 
\begin{equation} \nonumber
	-\sqrt{m} + Cm^{-1/6} < x_{m,- \lfloor m/2 \rfloor} < \cdots 
	< x_{m,-1} < x_{m,0} < x_{m,1} < \cdots <  x_{m,\lfloor m/2 \rfloor} \le \sqrt{m} - Cm^{-1/6}, 
\end{equation}
with a positive constant $C$ independent of $m$ (the Mhaskar-Rakhmanov-Saff number is 
$a_m(w_\beta) = \sqrt{m}$).

In the case (i),  the truncated Gaussian quadrature is defined by
\begin{equation} \nonumber 
	Q^{\rm{TG}}_{2 j(m)}f: = \sum_{1 \le |k| \le j(m)} \lambda_{m,k}(w_\beta) f(x_{m,k}), 
\end{equation} 
and in  the case (ii) by
\begin{equation} \nonumber 
	Q^{\rm{TG}}_{2 j(m)}f: = \sum_{0 \le |k| \le j(m)} \lambda_{m,k}(w_\beta) f(x_{m,k}),
\end{equation} 
where $\lambda_{m,k}(w_\beta)$ are the corresponding Cotes numbers. 

In the same ways, by using related results in \cite{MO2004} we can prove the following counterparts of Theorems \ref{thm:Q_n-d=1} and \ref{theorem:Int_n(W)}  for the unit ball 
$\bW^r_{1,w_\beta}\big(\mathring{\RR}^d\big)$ of the Markov-Sonin weighted Sobolev space $W^r_{1,w_\beta}\big(\mathring{\RR}^d\big)$ of mixed smoothness $r \in \NN$. 
\begin{theorem} \label{thm:Q_n-d=1,SMW}
	For any $n \in \NN$, let $m_n$ be the largest integer such that $2 j(m_n) \le n$. Then the quadratures	$Q^{\rm{TG}}_{2 j(m_n)} \in \Qq_n$, $n \in \NN,$ are  asymptotically optimal  for $\bW^r_{1,w_\beta}\big(\mathring{\RR}\big)$ and
	\begin{equation*}
		\sup_{f\in \bW^r_{1,w_\beta}\big(\mathring{\RR}\big)}
			\bigg|\int_{\RR} f(x) w(x) \rd x - Q^{\rm{TG}}_{2 j(m_n)}f\bigg| 
		\asymp
		\Int_n\big(\bW^r_{1,w_\beta}\big(\mathring{\RR}\big) \big)
		\asymp 
		n^{- r/2}.
	\end{equation*}
\end{theorem}

\begin{theorem} \label{theorem:Int_n(W)SMW}
	We have that
	\begin{equation*}\label{Int_n(W)}
		n^{-r/2} (\log n)^{(d-1)r/2} 
		\ll	
		\Int_n\bW^r_{1,w_\beta}\big(\mathring{\RR}^d\big)) 
		\ll 
		n^{-r/2} (\log n)^{(d-1)(r/2 + 1)}.
	\end{equation*}
\end{theorem}

\medskip
\noindent
{\bf Acknowledgments:}  
This research is funded by Vietnam Ministry of Education and Training under Grant No. B2023-CTT-08. A part of this work was done when  the author was working at the Vietnam Institute for Advanced Study in Mathematics (VIASM). He would like to thank  the VIASM  for providing a fruitful research environment and working condition. The author specially thanks Dr. Nguyen Van Kien for drawing Figures~\ref{Fig-1}~and~\ref{Fig-2}.
\bibliographystyle{abbrv}
\bibliography{AllBib}
\end{document}